\documentclass[11pt, a4paper]{amsart}
\usepackage{amsfonts,amssymb,amsmath,amsthm,amscd,mathtools,multicol,tikz, tikz-cd,caption,enumerate,mathrsfs,thmtools,cite}
\usepackage{inputenc}
\usepackage[foot]{amsaddr}
\usepackage[pagebackref=true, colorlinks, linkcolor=blue, citecolor=red]{hyperref}
\usepackage{latexsym}
\usepackage{fullpage}
\usepackage{microtype}
\usepackage{subfiles} 
\usepackage[lined,boxed,commentsnumbered]{algorithm2e}

\usepackage[T1]{fontenc}
\usepackage{palatino}
\makeatletter
\makeindex 
\newcommand{\be}{\begin{equation}}
\newcommand{\ee}{\end{equation}}
\newcommand{\beano}{\begin{eqn*}} 
	\newcommand{\eeano}{\end{eqnarray*}}
\newcommand{\ba}{\begin{array}}
	\newcommand{\ea}{\end{array}}

\declaretheoremstyle[headfont=\normalfont]{normalhead}

\newtheorem{theorem}{Theorem}[section]

\newtheorem{lemma}[theorem]{Lemma}

\newtheorem{definition}[theorem]{Definition}
\newtheorem{proposition}[theorem]{Proposition}
\newtheorem{remark}[theorem]{Remark}

\numberwithin{equation}{section}

\newtheorem{ex}[theorem]{Example}

\newtheorem{problem}[theorem]{Problem}
\@namedef{subjclassname@2020}{%
  \textup{2020} Mathematics Subject Classification}
\makeatother
\begin{document}
\title{Tensor product and quandle rings of connected quandles of prime order}
\author{Dilpreet Kaur}
\email{dilpreetkaur@iitj.ac.in}
\address{Indian Institute of Technology Jodhpur}

\author{Pushpendra Singh}
\email{singh.105@iitj.ac.in}
\address{Indian Institute of Technology Jodhpur}

\maketitle

\begin{abstract}
Let $\mathbb{C}$ be field of complex numbers and $X$ be a connected quandle of prime order. We study the regular representation of $X$ by describing the quandle ring $\mathbb{C}[X]$ as direct sum of right simple ideals. We provide description of tensor product of connected quandles of prime order. We further discuss multiplicity freeness of quandle ring decomposition for connected quandles of order $\leq 47$ and prove that $\mathbb{C}[X]$ decomposes multiplicity free for affine connected quandle $X$.
\end{abstract}

\section{Introduction}\label{1}

A quandle is a set with a binary operation that satisfies axioms representing algebraic versions of three Reidemeister moves of a knot diagram. Quandles were first studied independently by David Joyce \cite{DEJ82} and Sergei Matveev \cite{SVM82} as invariant of knots. They introduced the notion of the fundamental quandle of a knot and proved that it is an invariant of knots up to orientation. Since their introduction, quandles are being studied extensively to compute other invariants of knots and links.

Quandles are also being studied from a purely algebraic point of view. We briefly provide a summary of some results in this direction. Enumeration of quandles up to order 13 is provided in \cite{VY19}. A quandle is associated with groups in the form of an automorphism group, displacement group, and inner automorphism group. This further gives classes of quandles, such as homogeneous, connected, affine and latin quandles. Authors in \cite{HSV16} used properties of these groups to give various results related to connected quandles. For a prime $p$, description of connected quandles of order $p$ and $p^2$ is given in \cite{EGS01} and \cite{MG04} respectively. In fact, it is also shown that such connected quandles are affine quandles.

Analogous to group rings, quandle rings are introduced in \cite{BPS19}. Zero divisors and idempotents of quandle rings are studied in \cite{BPS22}. Authors in \cite{EFT19} studied the right ideals of quandle rings. They further provided results of isomorphic quandle rings for nonisomorphic quandles. Analogous to group representations, quandle representation is introduced in \cite{EM18}. A complete decomposition of the regular representation of dihedral quandles is provided in \cite{KS22} and \cite{ESZ23}. This result has been generalized to the class of Takasaki quandles in \cite{KS23}. Furthermore, in \cite{ESZ23}, it has been proven that Maschke's theorem is not valid in the case of quandle representations.

In this article, we study the quandle rings of connected quandles of prime order. In section \ref{2}, we give brief review on quandles and quandle rings. In section \ref{3}, we provide description of the inner automorphism group of prime connected quandles. In sections \ref{4} and \ref{5}, we describe the conjugacy classes and irreducible representations of the inner automorphism group. In section \ref{6}, we give complete decomposition of quandle ring of prime connected quandles as direct sum of right simple ideals. In section \ref{7}, we give the results for the tensor product of connected quandles of prime order. Furthermore, in section \ref{8}, we explore the decomposition of the quandle ring for connected quandles of order $\leq 47$ and prove the multiplicity freeness of quandle ring decomposition for affine connected quandles.

\section{Review of quandles and quandle rings} \label{2}
A quandle is an algebraic structure $(X,\triangleright)$ which satisfies the following three axioms:
\begin{enumerate}[(1).]
 \item $x \triangleright x = x,$  $~\forall~ x \in X.$
 \item $~\forall~x,y\in X,~\exists!~ z\in X$ such that $x=z \triangleright y.$
 \item $(x \triangleright y) \triangleright z = (x \triangleright z)\triangleright (y \triangleright z)$ $~\forall~x,y,z \in X.$
\end{enumerate}
Above axioms $(1),(2)$ and $(3)$ are algebraic versions of Reidemeister moves \text{I}$,$\text{II} and \text{III} respectively. If axiom $(1)$ is removed then the obtained structure is known as rack. The axioms $(2)$ and $(3)$ imply that the mapping $R_y : X \to X$ defined as $R_{y}(x) = x \triangleright y$ is an automorphism of quandle $X$ for each $y $ in $X$. They are known as inner automorphisms.
The axiom $(2)$ induces another binary operation ${\triangleright}^{-1}$ on $X$ defined as $x~{\triangleright}^{-1}~y = z$ if $x = z \triangleright y$. This gives the following alternate form of  second axiom 
$$ \displaystyle (x~{\triangleright}~y )~{\triangleright}^{-1}~y = x \quad \text{ and } \quad ( x~{\triangleright}^{-1}~y )~\triangleright~y = x  $$
for all $x,y \in X$.
The group generated by automorphisms of form $R_y,~y \in X$ is called inner automorphism group of quandle $X$. It is denoted by $Inn(X)$. Similar to the map $R_y,$ left multiplication map is defined as $L_{y}(x) = y \triangleright x$ for all $y \in X$. However $L_y$ need not be a bijective map.

If the action of group $Inn(X)$ on quandle $X$ is transitive then the quandle is called connected quandle. If $L_{y}$ is bijective map for all $y \in X,$ then the quandle is known as latin quandle. We refer to [\cite{BPS19},\cite{BPS22},\cite{EFT19}] for examples of quandles. In this article, we will focus on affine quandles, also known as Alexander quandles. We give their definition below.

\begin{definition}
Let $ \mathcal{M}$ be a $ \mathbb{Z}[t,t^{-1}]-$module with operation $\triangleright$ defined as $x \triangleright y = tx + (1-t)y$ for all $x,y \in \mathcal{M}$. Then $(\mathcal{M},\triangleright)$ is called an affine quandle.
\end{definition}
As an example trivial quandle of order $n$ is an affine quandle with $\mathcal{M} = \mathbb{Z}[t,t^{-1}]/(n,1-t)$. Alternatively affine quandle can be defined as pair $(A,f)$ where $A$ is an additive abelian group, $f \in Aut(A)$  and  quandle operation is defined as $x \triangleright y = f(x) +(id_{A}-f)(y)$ for all $x,y \in A$. For cyclic group $\mathbb{Z}_{m},~m \in \mathbb{N}$, we use symbol $(\mathcal{A}_{m},f)$ for the corresponding affine quandle. In particular, we use $(\mathcal{A}_{p},f)$ for affine quandles of prime order.

Let $(X,\triangleright)$ be a quandle and let $K$ be a ring. Then analogous to group ring, quandle ring is a $K-$vector space consisting of formal linear expressions of form $\sum_{x \in X}\alpha_{x}x$ with $x \in X$ and $\alpha_x \in K$ having all but finitely many $\alpha_x= 0$. It is denoted by $K[X]$ and it forms a ring structure with pointwise addition and the following multiplication
$$\left( \sum_{x \in X}a_{x}x \right) \cdot \left( \sum_{y \in X}b_{y}y \right) = \sum_{x,y \in X} a_{x}b_{y}(x \triangleright y)$$

We note that the map $\rho : Inn(X) \to GL(K[X])$ defined as $\rho(R_{x})(y) = y \triangleright x$ for all $R_{x} \in Inn(X), y \in X$ is a group representation of $Inn(X)$. Thus the decomposition of $K[X]$ as direct sum of right simple ideals is equivalent to the decomposition of $\rho$ as direct sum of irreducible representations of $Inn(X)$.

A representation of quandle $(X,\triangleright)$ is defined as a quandle homomorphism $\phi : X \to Conj(GL(V))$, where $Conj(GL(V))$ denotes the group of automorphism of vector space $V$ with conjugation as quandle operation. Hence we have $\phi(x \triangleright y) = \phi(y)\phi(x)\phi(y)^{-1}$ for all $x,y \in X$.(see \cite{EM18},\cite{ESZ23}).

\begin{definition}{Regular representation of quandle}: Let $X$ be a finite quandle and $\mathbb{C}[X]$ denote the quandle ring. The regular representation of $X$ is given by mapping $\phi : X \to Conj(GL(\mathbb{C}[X])$ such that $\phi_{z}(f)(y) = f(R_{z}(y))$ where $f = \sum_{x \in X}a_{x}x \in \mathbb{C}[X]$.

\end{definition}
The action in regular representation of quandle $X$ defined above is equivalent to the action in representation $\rho : Inn(X) \to GL(\mathbb{C}[X])$. This implies that restriction of  regular representation of quandle on subspace $W$ of quandle ring $\mathbb{C}[X]$ is a quandle representation if and only if $\rho|_W$ is a group representation of $Inn(X)$. This explains that quandle ring decomposition gives quandle representations.

\section{Inner automorphism groups of affine connected quandles } \label{3}
In this section, we give results regarding structure of inner automorphism group $Inn(\mathcal{A}_{m})$ of affine connected quandle $(\mathcal{A}_m,f)$. For cyclic group $\mathbb{Z}_m$, an automorphism $f \in Aut(\mathbb{Z}_{m})$ is defined as $f(x)=tx$ where $t \in \mathbb{Z}_{m}^{\ast}$. We note that affine connected quandles are latin \cite[ Corollary 3.1]{MB20}. Affine quandle $(A,f)$ is latin if and only if $1-f$ is bijective \cite[ Corollary 7.2]{HSV16}. We observe that if $gcd(m,2)=2$ then $1-f$ is not bijective for any $f \in Aut(\mathbb{Z}_{m})$. This implies if $(\mathcal{A}_{m},f)$ is connected quandle then $gcd(m,2)=1$. 

From now onwards, we assume that $1-f$ is a bijection. Let $n$ be the order of automorphism $f$, which means it is smallest $n \in \mathbb{N}$ such that $t^{n} \equiv 1~(mod~m)$. We use notation $R_j$ for right translations defined as $R_j(x) = tx+(1-t)j$ for each $j \in \mathcal{A}_m$. Now we have following result.
\begin{lemma} \label{InnAp}
Let $(\mathcal{A}_{m},f)$ denote the affine connected quandle of order $m$. Then $Inn(\mathcal{A}_{m}) = \mathbb{Z}_{m} \rtimes \mathbb{Z}_{n}$.
\end{lemma}

\begin{proof}
We prove that $Inn(\mathcal{A}_{m})$ has following presentation :
$$Inn(\mathcal{A}_{m}) = \{r,s~|~r^m=1,s^n=1,srs^{-1} = r^{t}\}$$

We claim that $r = R_{1}R_{0}^{n-1}$ and $s = R_{0}$. Using quandle operation, we get $R_{0}^{n}(x) = t^{n}x = x$ and $R_{1}R_{0}^{n-1}(x) = x +(1-t)$. Applying it $j$ times we get $\left( R_{1}R_{0}^{n-1} \right)^{j}(x) = x + j(1-t)$. Now since $(\mathcal{A}_{m},f) $ is connected so $j(1-t) \equiv 0~( mod~m) $ only when $j=m$. Thus order of $\left( R_{1}R_{0}^{n-1} \right)$ is $m$.

Now we prove $sr = r^{t}s$. First $R_{0}\left( R_{1}R_{0}^{n-1} \right)(x) = tx + t(1-t)$. Now $\left( R_{1}R_{0}^{n-1} \right)R_{0}(x) = \left( R_{1}R_{0}^{n-1} \right)(tx) = tx + (1-t)$. Applying it $t$ times we get $\left( R_{1}R_{0}^{n-1} \right)^{t}(tx) = tx + t(1-t)$. This proves the result.
\end{proof}

From now onwards, we follow the notation from the above lemma. We call $r$ and $s$ as generators of $Inn(\mathcal{A}_m)$ and  we note that $r=R_1R_0^{n-1}$ and $s=R_0$. The following result gives presentation of element $R_{j}^{k}$ for $1 \leq k \leq n-1$ in terms of generators of group $Inn(\mathcal{A}_{m})$.

\begin{lemma}\label{Rjk}
Let $R_{j}$ be right multiplication map for affine quandle $(\mathcal{A}_{m},f)$. Then for $1 \leq k \leq n-1$ we have
$$ R_{j}^{k} = \left( R_{1}R_{0}^{n-1} \right)^{j\sum_{i=0}^{k-1}t^{i} }R_{0}^{k} $$
\end{lemma}

\begin{proof}
We have $R_{j}(x) = tx+(1-t)j$, upon repeating this $k$ times, we get $R_{j}^{k}(x) = t^{k}x + (1+t+t^2+\hdots+t^{k-1})(1-t)j = \left( R_{1}R_{0}^{n-1} \right)^{j\sum_{i=0}^{k-1}t^{i} }R_{0}^{k}(x) $. This proves the result.
\end{proof}

\begin{remark} \label{fpr0}
We note that right translation map $R_{0}$ of quandle $(\mathcal{A}_m,f)$ is defined as $R_{0}(x) = tx$ which is same as automorphism $f$. Thus both $R_{0}$ and $f$ have same number of fixed points. For a group $G$, automorphism $f \in Aut(G)$ is called fixed point free if $f(x)=x$ implies $x=e$, where $e$ is identity of $G$.

\end{remark}

\begin{lemma}\label{fpfree}
For affine quandle $(\mathcal{A}_m,f)$, the automorphism $f$ is fixed point free if and only if $ \left( \mathcal{A}_{m}, f \right)$ is latin quandle.
\end{lemma}
\begin{proof}
To prove the forward direction, suppose $\left( \mathcal{A}_{m}, f \right)$ is not latin quandle. Then for $i \neq j$, we get $R_{i}(x) = R_{j}(x)$ for some $x$. This gives us $(1-t)i+tx = (1-t)j+tx$ and so  $t(i-j) = f(i-j) = i-j$. This implies $f$ is not fixed point free. The result now follows. 

Conversely suppose $f$ is not fixed point free automorphism. Then map $R_{0}$ has an element $j \neq 0$ such that $R_{0}(j)=j$. This implies $L_{j}(0) = j$ but we also have $L_{j}(j) = j$. So $L_{j}$ is not a permutation. This contradicts that $\left( \mathcal{A}_{m},f \right)$ is latin quandle. Hence $f$ is fixed point free automorphism.
\end{proof}

\begin{lemma}\label{prodcycles}
Each right multiplication map $R_{j}$ of an affine quandle $(\mathcal{A}_{p},f)$ is product of $(p-1)/n$ many disjoint cycles of order $n$.
\end{lemma}
\begin{proof}
From Lemma \cite[ Lemma 2.4]{LLV22}, we get that every right multiplication map of $(\mathcal{A}_m,f)$ have the same cycle structure. So it is enough to prove the result for map $R_{0}$. We observe that $\left(tx,t^{2}x,\hdots,t^{n}x \right)$ is a cycle in $R_{0}$. Now suppose $t^{i}x \equiv t^{j}x~(mod~p)$ then $p \mid x(t^i-t^j)$. Since $gcd(x,p)=1$ so $p \mid (t^i-t^j)$ i.e. $t^i \equiv t^j~(mod~p)$ which implies $i=j$. Thus we get that $R_{0}$ is product of $( p-1 )/n$ many disjoint cycles of order $n$.
\end{proof}
We give following example to note that Lemma \ref{prodcycles} does not hold for $(\mathcal{A}_m,f)$ for composite $m$.
\begin{ex} \label{example1}
Let $(\mathcal{A}_{21},f)$ be an affine connected quandle with $f$ given as $f(x)=11x$ then $R_{0}$ is given as permutation $(1,11,16,8,4,2)(3,12,6)(5,13,17,19,20,10)(7,14)(9,15,18)$. 
\end{ex}

\section{Conjugacy classes of inner automorphism group} \label{4}

In this section, we give description of conjugacy classes of group $Inn(\mathcal{A}_p)$ for prime $p$. Recall that $Inn(\mathcal{A}_m)$  has presentation of form $ \{ r,s ~|~r^{m}=1,s^{n}=1,srs^{-1} = r^{t} \}$. The order of group $Inn(\mathcal{A}_m)$ is $mn$ and each element of $Inn(\mathcal{A}_m)$ can be written as $s^{i}r^{j}$ for $ 0 \leq i \leq n-1, 0 \leq j \leq m-1$.  We begin with the following lemma.

\begin{lemma} \label{relations}
The elements of group $Inn(\mathcal{A}_{m})$ satisfy following properties:
\begin{itemize}
\item[$\left(a \right)$]
$s^{i}r^{j} = r^{jt^{i}}s^{i}$

\item[$\left(b \right)$]
$s^{-1}rs=r^{t^{-1}}$

\item[$\left(c \right)$]
$r^{j}s^{i} = s^{i}r^{jt^{-i}}$

\end{itemize}
\end{lemma}

\begin{proof}
We see that $s^{i}r^{j}=s^{i-1}(sr^{j}s^{-1})s = s^{i-1}(r^{jt})s = s^{i-2}(sr^{jt}s^{-1})s^{2} = s^{i-2}(r^{jt^{2}})s^{2}   $. Upon repeating the steps we get result $(a)$.
Now putting $i=1,j=t^{-1}$ in $(a)$ gives $(b)$. For $(c)$ observe that $r^{j}s^{i} = s(s^{-1}r^{j}s)s^{i-1} = s(r^{jt^{-1}})s^{i-1} = s^{2}(s^{-1}r^{jt^{-1}}s)s^{i-2} = s^{2}(r^{jt^{-2}})s^{i-2} $. Repeating the steps gives the result $(c)$.
\end{proof}

\begin{lemma} \label{C1}
The conjugacy class of element $r^{j} \in Inn(\mathcal{A}_{p})$ for $0 \leq j \leq p-1$ is given by $\{ r^{j},r^{jt},r^{jt^2},...$\\$...,r^{jt^{n-1}} \}$.
\end{lemma}
\begin{proof}
Let $s^{i}r^{k} \in Inn(\mathcal{A}_{p}) $. Then $s^{i}r^{k}r^{j}(s^{i}r^{k})^{-1} = s^{i}r^{k}r^{j}r^{-k}s^{-i}=s^{i}r^{j}s^{-i} = r^{jt^{i}}$ using Lemma \ref{relations}. We also have $jt^{x} \not \equiv jt^{y}~(mod~p)$ for $x \neq y$. Since order of $s$ is $n$, so $0\leq i \leq n-1$. Hence we get the result.
\end{proof}

\begin{lemma}\label{C2}
The conjugacy class of element $s^{i}r^{j} \in Inn(\mathcal{A}_{p})$ is given by $\{s^{i}r^{k}~|~k = 0,1,\hdots,p-1\}$ for each $1\leq i \leq n-1$ and $0 \leq j \leq p-1 $.

\end{lemma}

\begin{proof}
First, from Lemma \ref{relations}, we observe that if $i \neq j$ then $s^{i}$ and $s^{j}$ belong to different conjugacy classes. Now let $r^{k} \in Inn(\mathcal{A}_{p})$ where $0\leq k \leq p-1$ then by using Lemma \ref{relations}, we get $r^{k}s^{i}r^{-k} = r^{k(1-t^{i})}s^{i}$. Since $p$ is prime so $1-t^{i}$ is an automorphism of $\mathbb{Z}_{p}$ for all $1 \leq i \leq n-1$. The result now follows.
\end{proof}

Thus the group $Inn(\mathcal{A}_{p})$ have a conjugacy class $\{e\}$ of size $1$, $n-1$ conjugacy classes of size $p$ and $ \frac{p-1}{n}$ conjugacy classes of size $n$.

\begin{lemma}\label{sirj}
For each $1\leq k \leq p-1$, the set $\{R_{j}^{k}~|~ 0\leq j \leq p-1\}$ corresponds to the conjugacy class $\{s^{k}r^{j}~|~0 \leq j \leq p-1\}$ of $Inn(\mathcal{A}_{p})$.
\end{lemma}

\begin{proof}
Since $p$ is prime so $\sum_{i=0}^{k-1}t^{i}$ is an automorphism of $\mathbb{Z}_{p}$ for each $1 \leq k \leq p-1$. Now by using Lemma \ref{Rjk}, we get the result.
\end{proof}
We note that above lemma does not hold for quandle $(\mathcal{A}_m,f)$ with composite $m$ and same can be verified by considering Example \ref{example1}.

\section{Irreducible representations of $\mathbb{Z}_{p} \rtimes \mathbb{Z}_{n}$} \label{5}

Let $\mathbb{Z}_p$ denote the cyclic group of order $p$ and $u \in \mathbb{Z}_{p}^{\ast}$ be an element of order $n$. We denote semidirect product of $\mathbb{Z}_p$ and $\mathbb{Z}_n$ with symbol $\mathbb{Z}_p \rtimes_{\phi} \mathbb{Z}_n$ where $\phi$ is a group homomorphism defined as $\phi : \mathbb{Z}_n \to Aut(\mathbb{Z}_p)$ with $s \to \phi_s$ and $\phi_s : \mathbb{Z}_p \to \mathbb{Z}_p $ is defined as $\phi_s(r)=ur$. We observe that group $\mathbb{Z}_p \rtimes_\phi \mathbb{Z}_n$ has  presentation $\{r,s~\vert~ r^p,$\ $s^n,s^{-1}rs=r^u\}$. It follows from Lemma \ref{relations} that $u = t^{-1}$. For convenience, we write $\mathbb{Z}_p \rtimes \mathbb{Z}_n $ at the place of $\mathbb{Z}_p \rtimes_\phi \mathbb{Z}_n$.  We use Clifford theory for complete classification of irreducible representations of $\mathbb{Z}_p \rtimes \mathbb{Z}_n$.

\begin{definition}{Induced representation}: Let $G$ be a finite group and $K$ be a subgroup of $G$. Let $\phi: K \to GL(U)$ be a representation of $K$. Then induced representation of $(\phi,U)$ to $G$ is the representation $\rho : G \to GL(V)$ with 
$ V = \bigoplus_{i=1}^{n}g_{i}U $ where the set $\{g_i~|~i=1.2.\hdots,n\}$ is the transversal for $K$ in $G$.
\end{definition}

The group $G$ acts on each direct summand of $V$ in the following way :
$$g\cdot(g_{i}u) = g_{\sigma(i)}ku$$
where $g \in G$, $k \in K$, $u \in U$ and $\sigma \in S_{n}$.
We denote $Ind_{K}^{G}\phi$ with $\rho$. Note that we have $Dim~Ind_{K}^{G}{\phi} = (Dim~U) \cdot |G/K|$. 

Let $\hat{K}$ be the set consisting of irreducible representations of $K$. Then $g$-conjugate of representation $\phi$ for $g \in G,\phi \in \hat{K}$ is the representation $\prescript{g}{}{\phi} \in \hat{K}$ defined as
$$\prescript{g}{}{\phi}(k) = \phi(g^{-1}kg)$$
for all $k \in K$.
This gives rise to the following group known as inertia group $I_{G}(\phi)$ of representation $\phi~:$
$$I_{G}(\phi) = \{g \in G : \prescript{g}{}{\phi} \sim \phi\} $$
We note that  $K \subseteq I_{G}(\phi) \subseteq G$. 

The following theorem will be useful in constructing irreducible representations of $\mathbb{Z}_p \rtimes \mathbb{Z}_n$.

\begin{theorem}\cite[ Corollary 2.1]{SST09}\label{Conjugate}
 Let $K$ be normal subgroup of $G$. Then $Ind_{K}^{G} \phi $ is irreducible if and only if $I_{G}(\phi) = K$. Moreover, if $\phi, \phi_1 \in \hat{K}$ and $I_{G}(\phi) = I_{G}(\phi_{1}) = K$ then $Ind_{K}^{G}(\phi) \sim Ind_{K}^{G}(\phi_1)$ if and only if $\phi $ and $\phi_1$ are conjugate.
\end{theorem}

Since $\mathbb{Z}_p$ is a subgroup of $\mathbb{Z}_p \rtimes \mathbb{Z}_n$, we can use the theory discussed above to obtain irreducible representations of $\mathbb{Z}_p \rtimes \mathbb{Z}_n$ using the irreducible representations of $\mathbb{Z}_p$.
An  irreducible representation of $\mathbb{Z}_{p}$ is defined as
$$\phi_{k} :  \mathbb{Z}_{p} \to \mathbb{C}^* $$
and
$\phi_{k}({m}) = e^{{i2{\pi}mk}/p}$ where $0 \leq k \leq p-1$.

We recall that $\mathbb{Z}_{p} \rtimes \mathbb{Z}_{n}$ has finite presentation $ \left\{ r,s ~|~r^p=1,s^n=1,s^{-1}rs = r^{u}\right\}$. Let $\{1,s,s^2,\hdots s^{n-1}\}$ be the transversal for $\mathbb{Z}_{p}$ in $\mathbb{Z}_{p} \rtimes \mathbb{Z}_{n}$. Let $\phi_{k}$ denote the irreducible representation of $\mathbb{Z}_{p}$ defined above, then after induction of $\phi_k$ to $\mathbb{Z}_p \rtimes \mathbb{Z}_n$, we get the induced representation $\rho_{k}$ of $\mathbb{Z}_{p} \rtimes \mathbb{Z}_{n}$ as given below
$$ \rho_{k}(r)=\begin{pmatrix}
              {\omega}^{k} & 0 & \hdots & 0\\
              0 & {\omega}^{uk} & \hdots & 0\\
              \vdots & \vdots & \ddots & \vdots\\
              0 & 0 & \hdots &  {\omega}^{u^{n-1}k}
             \end{pmatrix}       \hspace{2cm}  \rho_{k}(s)=\begin{pmatrix}
                0 & 0 & 0 & \hdots & 1\\
              1 & 0 & & \hdots & 0\\
              \vdots & \vdots &  \vdots & \ddots & \vdots\\
              0 & 0 & \hdots & 1 & 0
             \end{pmatrix}  $$

where $ \omega = e^{{i}2{\pi}/p},~  u = t^{-1}$ and  $0 \leq k \leq p-1$.

\begin{proposition}\label{Formeta}
Let $G = H \rtimes \mathbb{Z}_{n}$ where $H = \mathbb{Z}_{p}$. Then
\begin{enumerate}
\item
If $\phi$ is a trivial irreducible representation of $H$ then $I_{G}(\phi) = G$.

\item
If $\phi $ is non trivial irreducible representation of $H$ then $I_{G}(\phi) = H$.

\end{enumerate}
\end{proposition}
\begin{proof}
First we see that $H \subseteq I_{G}(\phi_{k})$. Now we have $\prescript{s^j}{}{\phi_{k}}(r) = {\phi_{k}}(s^{-j}rs^{j}) = {\phi_{k}}(r^{u^{j}}) = {\phi_{u^{j}k}}(r) $. This implies
$$I_{G}(\phi_{k}) = \bigsqcup\limits_{j \in C_k}  Hs^{j}$$
where $C_{k} = \{j~:~u^{j}k \cong k~ mod~p\}$. For $k=0$, we get $I_{G}(\phi_{0}) = G$. For $k \neq 0$, using Fermat's little theorem, we get $j= p-1$ and hence $I_{G}(\phi_{k}) = H$.
\end{proof}

Since $\mathbb{Z}_p$ is normal subgroup of $\mathbb{Z}_p \rtimes \mathbb{Z}_n$, using Theorem \ref{Conjugate} and Proposition \ref{Formeta}, we get that $Ind_{H}^{G}\phi_{0}$ is reducible representation and $Ind_{H}^{G}\phi_k$ with $1\leq k \leq p-1$ are irreducible representations of $\mathbb{Z}_p \rtimes \mathbb{Z}_n$. 

Induced representation of trivial representation $\phi_0$ of $H$ splits into the following degree one irreducible representations
\begin{align*}
\zeta_{k} :& \mathbb{Z}_{p} \rtimes \mathbb{Z}_{n} \to \mathbb{C}^{\ast}\\
&\zeta_{k}(r) = 1\\
&\zeta_{k}(s) = e^{\frac{\textbf{i}2{\pi}k}{n}} \quad \quad k=0,1,...,n-1
\end{align*}

The remaining induced representations $Ind_{H}^{G}\phi_k$ with $1\leq k \leq p-1$ are not inequivalent representations. The following proposition is useful in determining inequivalent irreducible representations of $\mathbb{Z}_p \rtimes \mathbb{Z}_n$.
\begin{proposition}\label{Conjugategeneralised}
Let $G$ be group $\mathbb{Z}_{p} \rtimes \mathbb{Z}_{n}$. Let $\phi_{k}$ be complex irreducible representation of $\mathbb{Z}_{p}$. Then $\phi_{k}$ is conjugate to $\phi_{u^{i}k}$ where $1\leq i \leq n-1$.
\end{proposition}
\begin{proof}
We compute $\prescript{s^{i}}{}{\phi_{k}}(h) = \phi_{k}(s^{-i}hs^{i}) = \phi_{k}(h^{u^{i}}) = \phi_{ku^{i}}(h) $. Therefore $\phi_{k}$ and ${\phi_{ku^{i}}}$ are conjugate representations for $1\leq i \leq n-1$.
\end{proof}

Now together with Theorem \ref{Conjugate} and Proposition \ref{Conjugategeneralised}, we get that $\mathbb{Z}_p \rtimes \mathbb{Z}_n$ has $(p-1)/n$ inequivalent irreducible representations of degree $n$. Thus the group $Inn(\mathcal{A}_p)$ has $n$ degree one and $(p-1)/n$ degree $n$ inequivalent irreducible representations.

\section{Decomposition of quandle ring of connected quandles of prime order} \label{6}

In this section, we provide description of quandle ring of prime connected quandle as direct sum of right simple ideals. Let $\mathbb{C}$ denote the field of complex numbers and $(\mathcal{A}_{p}, f)$ denote the affine connected quandle of order $p$. We identify set $\mathcal{A}_{p}$ with $\{v_0, v_1, v_2, \dots , v_{p-1}\} \subseteq \mathbb{C}[\mathcal{A}_{p}]$, then $\mathcal{A}_{p}$ is basis of $\mathbb{C}[\mathcal{A}_{p}]$. It is easy to see that the module generated by  $v_{triv}= \sum_{i=0}^{p-1} v_{i}$ is simple module over $\mathbb{C}[\mathcal{A}_{p}]$, and we denote this module with $V_{triv}$. 

We use presentation $\{r,s~|~r^p,s^n,srs^{-1} = r^{t}\}$ of $Inn(\mathcal{A}_{p})$ by identifying $R_{1}{R_{0}}^{n-1}$ with $r$ and $R_{0}$ with $s$. The action of $Inn(\mathcal{A}_p)$ on $(\mathcal{A}_p,f)$ gives the representation $\phi_{\mathcal{A}_p} : Inn(\mathcal{A}_p) \to GL(\mathbb{C}[\mathcal{A}_p])$.

\begin{lemma}\label{Characterprimelatin}
Let $\chi_{\mathcal{A}_{p}}$ be character of the representation $\phi_{\mathcal{A}_{p}}$ defined as above. Then we have
$$\chi_{\mathcal{A}_{p}}(g)= \begin{cases}
                  p \quad \quad \text{ if } g=1\\
                  1 \quad \quad \text { if } g =s^{i}r^{j} \quad 1 \leq i \leq n-1, 0\leq j \leq p-1 \\
                  0 \text \quad \quad \text{ if } g =r^{j}  \quad  1 \leq j \leq p-1
                 \end{cases}$$
\end{lemma}

\begin{proof}
From Lemma \ref{C1} and \ref{C2}, we get that the group $Inn(\mathcal{A}_{p})$ has $n + (p-1)/n$ conjugacy classes, namely
$$\{1\}, \quad \{r^{j},r^{jt},r^{jt^2},\hdots,r^{jt^{n-1}}\}, \quad \{ s^{i}r^{j}; 1 \leq i \leq n-1, 0 \leq j \leq p-1\}.$$
Note that $\chi_{\mathcal{A}_p}(g) $ is equal to the number of fixed points of permutation $g \in Inn(\mathcal{A}_p)$. First, clearly $\chi_{\mathcal{A}_{p}}(1) = p $ and by using Lemma \ref{InnAp} we get that $r^{j}$ is a cycle of length $p$, so we have $\chi_{\mathcal{A}_{p}}(r^j) = 0$. From Lemma \ref{prodcycles} and \ref{sirj}, we have $\chi_{\mathcal{A}_{p}}(s^{i}r^{j}) = 1$
for all  $ 1 \leq i \leq n-1, 0 \leq j \leq p-1$.
\end{proof}

We denote the degree $n$ irreducible representations of $Inn(\mathcal{A}_{p})$ by the set $\{ \psi_j ~|~  1\leq j \leq (p-1)/n \}$ and the simple module associated with the representation $\psi_{j}$ by $V_j$. The following theorem gives decomposition of quandle ring $\mathbb{C}[\mathcal{A}_p]$ into direct sum of right simple ideals.

\begin{theorem} \label{Thm}
Let $\mathbb{C}$ be the field of complex numbers and $(\mathcal{A}_{p},f)$ denote the connected quandle of order $p$. Then
$$\mathbb{C}[\mathcal{A}_{p}] = V_{triv} \oplus \bigoplus_{j=1}^{(p-1)/n} V_{j}$$
\end{theorem}

\begin{proof}
We prove the result by showing that the representation $\psi_j$ has multiplicity one in representation $\phi_{\mathcal{A}_p}$ for all $1 \leq j \leq (p-1)/n$.  Let $\chi_{j}$ be the character of representation $\psi_{j}$  for $1 \leq j \leq (p-1)/n$, then by using Lemma \ref{Characterprimelatin} and representation theory of $Inn(\mathcal{A}_p)$ from section \ref{5}, we compute

 $$ \langle \chi_{\mathcal{A}_{p}}, \chi_j \rangle = \frac{1}{np}\sum_{g\in Inn(\mathcal{A}_{p})} \chi_{\mathcal{A}_{p}}(g) \chi_j(g)
  = 1.$$
Now by using \cite[Th. 14.17]{JL01}, we get that for each $1 \leq j \leq (p-1)/n$, $V_{j}$ appears in $\mathbb{C}[\mathcal{A}_{p}]$ with multiplicity one. On verifying dimensions on both sides, we get 
$$\mathbb{C}[\mathcal{A}_{p}] = V_{triv} \oplus \bigoplus_{j=1}^{(p-1)/n} V_{j}$$

\end{proof}

\section{Tensor product of connected quandles of prime order}  \label{7}

In this section, we describe the tensor product of connected quandles of order $p$. The concept of tensor product of quandles is introduced in \cite{SK21}. Let $G$ be a group and let $X,Y$ be quandles equipped with the right and left $G-$actions respectively. Now define the following equivalence relation on $X \times Y$ 
$$\left( x\cdot g,y \right) \sim (x,g\cdot y) \quad \text{ for } x \in X,y \in Y \text{ and } g \in G$$
Then the tensor product $X {\otimes}_{G} Y$ is defined as the set of all equivalence classes $[(x,y)]$ of above defined action on the set $X \times Y$. The canonical tensor product of quandles is studied in \cite{SK21}. We study the same in this article.  We recall the definition below.

Let $X$ be a quandle and $F(X)$ denote the free group generated by $X$. The right and left action of $F(X)$ on $X$ is defined as
$$x \cdot g = (((x~{\triangleright}^{{\epsilon}_{1}}~x_{1})~{\triangleright}^{{\epsilon}_{2}}~x_{2})\hdots)~{\triangleright}^{{\epsilon}_{m}}~x_{m}$$

$$g \cdot x = x \cdot g^{-1} = (((x~{\triangleright}^{-\epsilon_m}~x_m)~{\triangleright}^{-\epsilon_{m-1}}~x_{m-1})\cdots )~\triangleright^{-\epsilon_1}~x_1 $$

respectively, where $x \in X$ and $ g = {x_1}^{{\epsilon}_{1}}x_{2}^{{\epsilon}_{2}}\cdots{x_m}^{{\epsilon}_{m}} \in F(X)$.
 
The above defined action is known as the canonical group action of $F(X)$ on $X$ and $X \otimes_{F(X)} X$ is called the canonical tensor product. For convenience, we omit the subscript and use symbol $X \otimes X$ to denote the tensor product.

We note that the above defined action of $F(X)$ on $X$ has the same orbit structure as to the natural action of $Inn(X)$ on $X$. Now, by using \cite[Equation 12]{SK21}, we get that the tensor product $X \otimes X$ is the set of orbits of $X \times X$ under the right action of $Inn(X)$.

Let $\tau : X \times X \to X \times X$ be the involution map defined as $\tau(x,y) = (y,x)$ for all $x,y \in X$. This map induces an involution of $X \otimes X$ as $\tau[( x,y )] = [( y,x )]$. We use the $X \otimes X/ \langle \tau \rangle$ to denote the equivalence classes of $X \otimes X$ modulo equivalence relation of $\tau$. Now we recall some results related with transitive group action. We begin with the definition of rank of transitive group action.

\begin{definition}
Let $ \alpha : G \times X \to X$ be a transitive group action. This induces group action on $X \times X$ in following way
$${\alpha_{g}^{2}}(x,y) = (\alpha_{g}(x),\alpha_{g}(y)) \quad \quad \text{ for all }(x,y) \in X \times X.$$
The number of orbits of the group action $\alpha^{2}$ is called the rank of $\alpha$.
\end{definition}

\begin{lemma}\cite[Corollary 7.2.10]{BS11}\label{BS7.1}
Suppose $\alpha : G \times X \to X$ is a transitive group action. Let $\tilde{\alpha} : G \to GL(C[X])$ be the associated  permutation representation and $\chi_{\tilde{\alpha}}$ be the character associated to $\tilde{\alpha}$. Then
$$\text{rank}(\alpha) = \langle \chi_{\tilde{\alpha}},\chi_{\tilde{\alpha}} \rangle  $$

\end{lemma}

\begin{remark} \label{remark}

Let $G$ be finite group and $\phi_{1},\phi_{2},...,\phi_{n} $ be the irreducible representations of $G$. Then we have $\tilde{\alpha} = {m_1}\phi_1 \oplus {m_2}\phi_2 \oplus \cdots \oplus {m_n}\phi_n $ where $m_{i} \in \mathbb{Z}_{\geq 0}$. Then using Lemma \ref{BS7.1}, we get
$$\text{rank}(\alpha) = \sum_{i=1}^{n} m_{i}^{2}$$

\end{remark}

The following result gives the size of the tensor product of connected quandles of prime order using the decomposition of quandle rings $\mathbb{C}[\mathcal{A}_{p}]$. Recall that we denote the order of $f \in Aut(\mathbb{Z}_{p})$ with $n$.

\begin{theorem}
Let $(\mathcal{A}_{p},f)$ be the connected quandle of order $p$. Then 
$$|\mathcal{A}_{p} \otimes \mathcal{A}_{p}| = 1+(p-1)/n$$

\end{theorem}

\begin{proof}
Using Theorem \ref{Thm}, we get that $m_i =0$ or $m_i = 1$. Furthermore, we get $\sum_{i}m_{i}^2 = 1 + (p-1)/n$. Now the result follows from Remark \ref{remark}.
\end{proof}
Now we provide description of elements of $\mathcal{A}_{p} \otimes \mathcal{A}_{p}$. We begin with defining the following two types of sets

$$A(k) = \{(i,i+tk),(i,i+{t^2}k)\hdots,(i,i+{t^n}k)~|~ i \in \mathcal{A}_{p}\}$$
and 
$$\tilde{A}(k) = \{(i+tk,i),(i+{t^2}k,i)\hdots,(i+{t^n}k,i)~|~i \in \mathcal{A}_{p}\}$$
for $k = 0,1,\hdots,p-1$. 

For $k=0$, we have $A(0) = \tilde{A}(0)$. We name elements of form  $(i,i)$ as diagonal elements and $A(0)$ as diagonal orbit. Now, we discuss some examples of the tensor product.
\begin{ex}
Let $X = (\mathcal{A}_{13},f)$ be affine quandle with $t=8$, then elements of $X \otimes X$ have following properties:
\begin{equation*}
\begin{array}{ll}
~A(0) = [(0,0)] &  \\ 
~A(1) = [(0,1)]  = [(0,5)] = [(0,8)]  = [(0,12)]\\
~A(2) = [(0,2)]  = [(0,3)] = [(0,10)] = [(0,11)] \\
~A(4) = [(0,4)]  = [(0,6)] = [(0,7)]  = [(0,9)]\\ 
\end{array}
\end{equation*}
We observe that order of $f$ is $4$, which is even. We have $A(1) = \tilde{A}(1),A(2) = \tilde{A}(2),A(4) = \tilde{A}(4)$. This gives
\begin{equation*}
X \otimes X =\{ [(0,0)],[(0,1)],[(0,2)],[(0,4)] \}. 
\end{equation*}
and
\begin{equation*}
X \otimes X / \langle \tau \rangle =\{ [(0,0)],[(0,1)],[(0,2)],[(0,4)] \}. 
\end{equation*}
\end{ex}

\begin{ex}
Let $X = (\mathcal{A}_{13},f)$ be affine quandle with $t=9$. Then elements of $X \otimes X$ have following properties:
\begin{equation*}
\begin{array}{ll}
~A(0) = [(0,0)] &  \\ 
~A(1) = [(0,1)]  = [(0,3)]  = [(0,9)] \\
~A(2) = [(0,2)]  = [(0,5)]  = [(0,6)] \\
~A(4) = [(0,4)]  = [(0,10)] = [(0,12)] \\
~A(7) = [(0,7)]  =  [(0,8)] = [(0,11)]\\ 
\end{array}
\end{equation*}
In this case order of $f$ is $3$, which is odd. We have $A(1) \neq \tilde{A}(1) ,A(2) \neq \tilde{A}(2),A(4) \neq \tilde{A}(4),A(7) \neq \tilde{A}(7)$. We have $\tilde{A}(1) = A(4),\tilde{A}(2) = A(7)$. This gives 
\begin{equation*}
X \otimes X =\{ [(0,0)],[(0,1)],[(0,2)],[(0,4)],[(0,7)] \}. 
\end{equation*} 
and
\begin{equation*}
X \otimes X / \langle \tau \rangle =\{ [(0,0)],[(0,1)],[(0,2)] \}. 
\end{equation*}
\end{ex}

Similar to the case of dihedral quandles in \cite{SK21}, we define $d_{p} : \mathcal{A}_{p} \times \mathcal{A}_{p} \to \mathbb{Z}_{\geq 0}$ to be the distance function on $\mathcal{A}_{p} = \mathbb{Z}_{p}$ induced from the distance on $\mathbb{Z}$ given by $d_{0} : \mathbb{Z} \times \mathbb{Z} \to \mathbb{Z}_{\geq 0} $ with $d_{0}(x,y) = |x-y|$. Let $g \in Inn(\mathcal{A}_{p})$ then using Lemma \ref{InnAp} and Lemma \ref{sirj}, we have either $g = R_{j}^{l}$ for some $1 \leq l \leq n-1$ and $0 \leq j \leq p-1$ or $g= \left( {R_1}{R_0}^{n-1} \right)^{j'}$ for $ 0 \leq j' \leq p-1$. Now we have 
\begin{equation}\label{eqdist1}
d_{p}(x\cdot g,y\cdot g) = t^{l}d_{p}(x,y)  \quad \text{ for any } x,y \in \mathcal{A}_{p} \text{ and } g = R_{j}^{l} \in Inn(\mathcal{A}_{p})
\end{equation}
and 
\begin{equation} \label{eqdist2}
d_{p}(x\cdot g,y\cdot g) = d_{p}(x,y)  \quad \text{ for any } x,y \in \mathcal{A}_{p} \text{ and } g = \left( {R_1}{R_0}^{n-1} \right)^{j'} \in Inn(\mathcal{A}_{p})
\end{equation}

\begin{lemma} \label{distance}
Let $( \mathcal{A}_{p},f)$ be the connected quandle of prime order with $gcd(n,2)=2$. Let $k \in \{0,1,\hdots p-1\}$. If $d_{p}(x,y) = t^{i}k$ for some $1 \leq i \leq n$  then there exist an element $g \in Inn(\mathcal{A}_{p})$ with $x \cdot g = 0 $ and $y\cdot g =k$.

\begin{proof}
Since ( $\mathcal{A}_{p},f)$ is connected quandle, for any $x \in \mathcal{A}_{p}$ there exist an element $g \in Inn(\mathcal{A}_{p})$ such that $x \cdot g =0$. The result is true for $k=0$.

Now suppose $0 < k \leq p-1$ and $d_{p}(x,y) = t^{i}k$. Let $g' = \left( {R_1}{R_0}^{n-1} \right)^{j'}  \in Inn(\mathcal{A}_{p})$ be an element with $x\cdot g' = 0$. Such $g'$ exist because $g'$ is a cycle of length $p$. From equation \ref{eqdist2}, we have $d_{p}(x \cdot g',y \cdot g') = t^{i}k$. Choose $g''= R_{0}^{l}$ such that we have $i+l=n$. Now let $g_1=g''g'$, then by equation \ref{eqdist1}, we have either $y\cdot g_1 = t^{i+l}k=k $ and so  $g = g_1$ or $y\cdot g_1 = -t^{i+l}k=-k$ then $g = g_1R_{0}^{m}$, where $R_{0}^{m}(-k) = k$. Such $m$ exists if $gcd(n,2) = 2$. Thus in this case $A(k) = \tilde{A}(k)$ and $A(k) \subseteq [(0,k)]$ for each $k = 0,1,\hdots,p-1$. If $gcd(n,2)=1$, then such $m$ does not exists and we get $A(k) \subseteq [(0,k)]$ and $\tilde{A}(k) \subseteq [(k,0)]$.
\end{proof}

\end{lemma}

\begin{theorem}
Let ( $\mathcal{A}_{p},f)$ be the affine connected quandle of order $p$. Then the tensor product $\mathcal{A}_{p} \otimes \mathcal{A}_{p}$ consists of $1+(p-1)/n$ elements of form  $A(k), k \in \mathcal{A}_{p}$ where
$$A(k) = \{(i,i+tk),(i,i+t^{2}k)\cdots (i,i+t^{n}k)~|~i \in \mathcal{A}_{p}\}$$
\end{theorem}

\begin{proof}
It is easy to see that $A(0)$ is an element of $\mathcal{A}_p \otimes \mathcal{A}_p$ and also $A(0),A(1),\hdots,$\ $A(p-1)$ cover $\mathcal{A}_{p} \times \mathcal{A}_{p}$. From Lemma \ref{distance} and equation \ref{eqdist1}, we get that $A(k)$ is an element of $\mathcal{A}_p \otimes \mathcal{A}_p$. However for $k \neq k'$, it is possible that $A(k) =  A(k') $. We can choose  $(p-1)/n$ number of distinct $A(k)$'s which cover $\mathcal{A}_p \otimes \mathcal{A}_p$. The set of such $k$ values is equal to set of coset representatives of subgroup $K= \langle t \rangle$ of group $\mathbb{Z}_{p}^{\ast}$. Since $A(k)$ has elements with distance $d(x,y)=t^{l}k, 1 \leq l \leq n$, so if $k_1,k_2 \in \mathbb{Z}_{p}^{\ast} / K$ such that $k_1K \neq k_2K$ then we have $A(k_1) \neq A(k_2)$.  

\end{proof}

\begin{theorem}
Let ( $\mathcal{A}_{p},f)$ be the connected quandle of order $p$ with $gcd(n,2)=2$. Then the tensor product $\mathcal{A}_{p} \otimes \mathcal{A}_{p} / \langle \tau \rangle$ consists of $1+(p-1)/n$ elements of form  $A(k)$.
\end{theorem}

\begin{proof}
 For $gcd(n,2) =2$, using  Lemma \ref{distance} and equation \ref{eqdist1}, we get that $[(0,k)] = [(k,0)]$ and so $A(k) = \tilde{A}(k)$. This implies $\tau(A(k)) = A(k)$ for all $k \in \mathcal{A}_p$. Therefore $ \mathcal{A}_{p} \otimes \mathcal{A}_{p}  = \mathcal{A}_{p} \otimes \mathcal{A}_{p} / \langle \tau \rangle$.
\end{proof}

\begin{theorem}
Let ( $\mathcal{A}_{p},f)$ be the connected quandle of order $p$  with $gcd(n,2)=1$. Then the tensor product $\mathcal{A}_{p} \otimes \mathcal{A}_{p} / \langle \tau \rangle$ consists of $A(0)$ and $(p-1)/2n$ elements of form $A(k)$.
\end{theorem}

\begin{proof}
Since $gcd(n,2)=1$, so using Lemma \ref{distance}, and equation \ref{eqdist1}, we get that  $[(0,k)] \neq [(k,0)]$ and so $A(k) \neq \tilde{A}(k)$ for $k \neq 0$. This gives us $\tau(A(k)) = \tilde{A}(k)$ and hence $\mathcal{A}_{p} \otimes \mathcal{A}_{p} / \langle \tau \rangle$ consists of $A(0)$ and $(p-1)/2n$ elements of form $A(k)$.
\end{proof}

\section{Multiplicity freeness of quandle ring decomposition of connected quandles} \label{8}

Let $G$ be a finite group and $X$ be a finite set. Suppose $G$ acts on $X$ transitively. This gives group homomorphism $\rho : G \to GL(\mathbb{C}[X])$.

\begin{definition}
Let $\phi : G \to GL(V)$ be a representation of group $G$. The representation $\phi$ is called multiplicity free if it contains irreducible representations of $G$ with multiplicity at most one. In other words $\phi = \oplus_{i=1}^{n} \phi_i$ where $\phi_i$ irreducible and $\phi_i \nsim \phi_j$ for $i \neq j$.
\end{definition}
The study of multiplicity freeness of representation $\rho$ is an important area in group theory. We refer to \cite[Chapter 4]{SST08} for more details. 

Suppose a group $G$ acts on sets $X$ and $Y$. We say the sets $X$ and $Y$ are isomorphic as $G-$sets if there is a bijective map $f : X \to Y$ satisfying $f(g\cdot x) = g \cdot f(x)~\forall~g \in G,~x \in X$.

\begin{lemma}\cite[Lemma 3.1.6]{SST08}
Let group $G$ acts on $X$ transitively and $K$ denotes the stabilizer of an arbitrary $x_0 \in X$. Then $X$ and $G / K$ are isomorphic as $G-$sets.
\end{lemma}

\begin{remark} \label{rmeq}
It follows from above lemma that representation $\rho : G \to GL(\mathbb{C}[X])$ is equivalent to representation $\rho' : G \to GL(\mathbb{C}[G/K])$.

\end{remark}

\begin{definition}{Double Coset}: Let $H$ and $K$ be subgroups of $G$.  Then for each $g$ in $G$, the set $HgK = \{hgk~|~h \in H, k \in K\}$ is called an $(H,K)-$ double coset of $g$. We use symbol $H \backslash G/K$ to denote the set of all $(H,K)-$double cosets.
\end{definition}

\begin{definition}{Gelfand Pair}: Let $G$ be a group and $K$ be a subgroup of $G$. The pair $(G,K)$ is called Gelfand pair if the algebra $\mathbb{C}[K\backslash G/K]$ is commutative.

\end{definition}
The following theorem gives the relation between Gelfand pair and multiplicity freeness of the corresponding representation of group $G$.

\begin{theorem}\cite[Theorem 4.4.2]{SST08} \label{mf}
Let $G$ be a finite group and $K$ be a subgroup of $G$. Then the following are equivalent
\begin{itemize}
\item[(i)]
$(G,K)$ is a Gelfand pair.
\item[(ii)]
$Hom_G(\mathbb{C}[G/K],\mathbb{C}[G/K])$ is commutative.
\item[(iii)]
The representation $\rho' : G \to GL(\mathbb{C}[G/K])$ is multiplicity free.
\end{itemize}
\end{theorem}

Recall that a connected quandle is a quandle $X$ such that the action of $Inn(X)$ on $X$ is transitive.  We prove that for an affine connected quandle $X$, quandle ring $\mathbb{C}[X]$ decomposes multiplicity free. We begin with the following result. 

\begin{lemma}\cite[ Theorem 6.1]{BDS17}
Let $X=(A,f)$ denote the affine connected quandle. Then $Inn(X) \cong A \rtimes \langle f \rangle$.
\end{lemma}

\begin{definition}
A quandle $X$ is called homogeneous quandle if the action of group $Aut(X)$ on $X$ is transitive.
\end{definition}
Let $G$ be a group, $\phi \in Aut(G)$ and $H$ is a subgroup of $ C_{G}(\phi) = \{g \in G~:~\phi(g)=g\}$, then the set of right cosets $G/H$ forms a quandle with operation  $Hx \triangleright Hy = H\phi(xy^{-1})y$ and it is denoted by $Q_{Hom}(G,H,\phi)$. It is proved in \cite{HSV16} that a quandle is homogeneous if and only if it is of form $Q_{Hom}(G,H,\phi)$.

We note that connected quandles are homogeneous quandles. Moreover affine quandle $(A,f)$ is homogeneous quandle of form $Q_{Hom}(A,\{0\},f)$.

In the case, when $G$ acts transitively of set $X$, the quandle $Q_{Hom}(G,H,\phi)$ can be stated as of form $Q_{Hom}(G,G_e,\phi)$ for some $e \in X$ and  stabilizer $G_e$ of $e$. If a connected quandle $X$ is represented as homogeneous quandle using $G = Inn(X)$ then such a representation is called canonical representation. (See \cite{HSV16} for more details.) 

We recall that semidirect product of groups $G$ and $K$  with $G$ as normal subgroup is defined as $G \rtimes K$ with multiplication $(g_1,k_1)(g_2,k_2) = (g_1\psi_{k_1}(g_2),k_1k_2)$ where $\psi : K \to Aut(G)$ is a group homomorphism.

\begin{lemma}\label{GP}
Let $G$ be an abelian group then $(G \rtimes K,K)$ is a Gelfand pair.
\end{lemma}

\begin{proof} 
We prove the result by showing commutativity of the double cosets i.e. $K(g_1,1)K(g_2,1)K=K(g_2,1)K(g_1,1)K$ for all $g_1,g_2 \in G$. Let $k_1,k_2,k_3 \in K$. Then we have
\begin{align*}
(1,k_1)(g_1,1)(1,k_2)(g_2,1)(1,k_3) &= (1,k_1)(g_1,1)({\psi}_{k_2}(g_2),k_2)(1,k_3) \\
&= (1,k_1)(g_1{\psi}_{k_2}(g_2),k_2)(1,k_3)\\
&=(1,k_1)({\psi}_{k_2}(g_2)g_1,k_2)(1,k_3)\\
&=(1,k_1)(1,k_2)(g_2,1)(1,k_{2}^{-1})(g_1,1)(1,k_2)(1,k_3)\\
\end{align*}
This implies $ K(g_1,1)K(g_2,1)K \subseteq K(g_2,1)K(g_1,1)K$. Similarly we have 
\begin{align*}
(1,k_1)(g_2,1)(1,k_2)(g_1,1)(1,k_3) &= (1,k_1)(g_2,1)(\psi_{k_2}(g_1),k_2)(1,k_3) \\
&= (1,k_1)(g_2{\psi}_{k_2}(g_1),k_2)(1,k_3) \\
&= (1,k_1)({\psi}_{k_2}(g_1)g_2,k_2)(1,k_3)\\
&= (1,k_1)(1,k_2)(g_1,1)(1,k_2^{-1})(g_2,1)(1,k_2)(1,k_3)
\end{align*}
This implies $K(g_2,1)K(g_1,1)K  \subseteq K(g_1,1)K(g_2,1)K$. Hence $(G \rtimes K,K))$ is a Gelfand pair.
\end{proof}

\begin{theorem} \label{affmultfree}
Let $(A,f)$ be affine connected quandle of finite order. Then quandle ring $\mathbb{C}[A]$ decomposes multiplicity free as direct sum of right simple ideals. 
\end{theorem}

\begin{proof}
Let $K = \langle f \rangle$ then $Inn(A) \cong A \rtimes K$. Moreover $Inn(A)$ consists of maps of form $x \to e + f^{n}(x)$ where $e \in A, n \in \mathbb{N}$. Let $0 \in A$ be the identity element. Then we get $Inn(A)_{0} \cong K$. This implies the quandle $(A,f)$ has canonical representation of form $Q_{Hom}(A \rtimes K,K,\phi)$. Using Remark \ref{rmeq}, we get that the representation $\rho : Inn(A) \to GL(\mathbb{C}[A])$ is equivalent to the representation $ {\rho}{'} : Inn(A) \to GL(\mathbb{C}[{A\rtimes K} / K])$. Now we get the result using Lemma \ref{mf} and Lemma \ref{GP}.
\end{proof}
The following theorem gives a class of connected quandles for which the quandle ring decomposes multiplicity free.
\begin{theorem}
Let $X$ be a finite connected quandle such that $X \otimes X = X \otimes X / \langle \tau \rangle$ then $\mathbb{C}[X]$ decomposes multiplicity free.
\end{theorem}
\begin{proof}
Since $X \otimes X = X \otimes X / \langle \tau \rangle$. This implies $\tau[(x,y)] = [(y,x)] = [(x,y)]$ for all $x,y \in X$. Now the result follows from \cite[Example 4.3.2,Lemma 4.3.4]{SST08}.
\end{proof}

We note that the analogue of Theorem \ref{affmultfree} for connected quandles is not true, as evident from the following example.
\begin{ex}
Let $( X,\triangleright)$ be a quandle of order $12$ with multiplication table as below

$$(X,\triangleright)=\begin{array}{|c|c c c c c c c c c c c c|} 
		\hline
\triangleright \ &  e_1 & e_2 & e_3 & e_4 & e_5 & e_6 &  e_7  & e_8 & e_9 & e_{10} & e_{11} & e_{12}\\ \hline
\ e_1 &  e_1 & e_{11} & e_{10} & e_5 & e_{10} & e_9 & e_{11} & e_9 & e_1 & e_5 & e_2 & e_2 \\
\ e_2 &  e_{11} & e_2 & e_8 & e_7 & e_8 & e_3 & e_2 & e_3 & e_{11} & e_7 & e_1 & e_1 \\
\ e_3 &  e_4 & e_8 & e_3 & e_9 & e_3 & e_2 & e_8 & e_2 & e_4 & e_9 & e_5 & e_5 \\
\ e_4 &  e_3 & e_{10} & e_9 & e_4 & e_9 & e_{11} & e_{10} & e_{11} & e_3 & e_4 & e_6 & e_6 \\
\ e_5 &  e_{10} & e_6 & e_5 & e_1 & e_5 & e_7 & e_6 & e_7 & e_{10} & e_1 & e_3 & e_3 \\
\ e_6 &  e_8 & e_5 & e_7 & e_{11} & e_7 & e_6 & e_5 & e_6 & e_8 & e_{11} & e_4 & e_4 \\
\ e_7 &  e_{12} & e_7 & e_6 & e_2 & e_6 & e_5 & e_7 & e_5 & e_{12} & e_2 & e_9 & e_9 \\
\ e_8 &  e_6 & e_3 & e_2 & e_{12} & e_2 & e_8 & e_3 & e_8 & e_6 & e_{12} & e_{10} & e_{10} \\
\ e_9 &  e_9 & e_{12} & e_4 & e_3 & e_4 & e_1 & e_{12} & e_1 & e_9 & e_3 & e_7 & e_7 \\
\ e_{10} &  e_5 & e_4 & e_1 & e_{10} & e_1 & e_{12} & e_4 & e_{12} & e_5 & e_{10} & e_8 & e_8 \\
\ e_{11} &  e_2 & e_1 & e_{12} & e_6 & e_{12} & e_4 & e_1 & e_4 & e_2 & e_6 & e_{11} & e_{11} \\
\ e_{12} &  e_7 & e_9 & e_{11} & e_8 & e_{11} & e_{10} & e_9 & e_{10} & e_{7} & e_8 & e_{12} & e_{12} \\
	
		 \hline
		\end{array}\\$$

We observe that $(X,\triangleright)$ is non affine connected quandle with $Inn(X) \cong S_{4}$. Upon decomposing  $\mathbb{C}[X]$ we get that $\mathbb{C}[X] \cong V_{(4)} \oplus V_{(2,2)} \oplus V_{(2,1,1)} \oplus V_{(3,1)} \oplus V_{(3,1)}$. Here $V_{\lambda}$ represents the irreducible representation of $S_{4}$ with respect to partition $\lambda$. In fact, this is the smallest order connected quandle for which quandle ring decomposition is not multiplicity free.
\end{ex}
Furthermore, the analogue of Theorem \ref{affmultfree} is also not true for latin quandles. All connected quandles upto order $47$ are available in  \texttt{GAP} \cite{GAP} package \texttt{RIG} \cite{RIG}. We note that the quandles in \texttt{RIG} package are left distributive. We use the right distributive versions available in \cite{QC}. We follow Rig id convention given in \cite{QC} in this article.

Let $X$ be connected quandle of order $36$ with Rig id $[36,59]$. We note that $X$ is non affine latin quandle with $Inn(X) \cong ((\mathbb{Z}_{3} \times \mathbb{Z}_{3}) \rtimes Q_{8}) \rtimes \mathbb{Z}_{3} $. It can be verified using \texttt{GAP} that quandle ring $\mathbb{C}[X]$ does not decompose multiplicity free for this quandle. We also note that this is the only latin quandle of order $\leq 47$ for which $\mathbb{C}[X]$ does not have multiplicity free decomposition. 

We provide the list of connected quandles upto order $47$ having non multiplicity free $\mathbb{C}[X]$ decomposition in the following table.

\begin{center}
    \begin{tabular}{|c|c|}
        \hline
        Order & Rig id \\
        \hline
        $12$ & $[12,1]$ \\
        \hline
        $15$ & $[15,2]$ \\
        \hline
        $20$ & $[20,1]$ \\
        \hline
        $21$ & $[21,6]$ \\
        \hline
        $24$ & $[24,1],[24,3],[24,4],[24,5],[24,6],[24,7],[24,9],[24,10],[24,11],[24,20]$ \\
        \hline
        $27$ & $[27,1]$ \\
        \hline
        $30$ & $[30,1],[30,3],[30,4],[30,5]$ \\
        \hline
        $32$ & $[32,1],[32,2],[32,3]$ \\
        \hline
        $36$ & $[36,1],[36,4],[36,7],[36,9],[36,17],[36,20],[36,21],[36,27],[36,28],[36,29] [36,30],$\\
        &$[36,58],[36,59],[36,61],[36,62],[36,63]$ \\
        \hline
        $40$ & $[40,1],[40,7],[40,8],[40,9],[40,10],[40,11],[40,19],[40,20]$ \\
        \hline
        $42$ & $[42,11],[42,12],[42,13],[42,14],[42,24]$ \\
        \hline
        $45$ & $[45,27],[45,28],[45,29]$ \\
        \hline
    \end{tabular}
    \captionof{table}{List of connected quandles with non multiplicity free $\mathbb{C}[X]$ decomposition } \label{table:1}
\end{center}

So far we have provided two classes of finite connected quandles with multiplicity free $\mathbb{C}[X]$ decomposition. However there are connected quandles which do not belong to any of these two classes but have multiplicity free $\mathbb{C}[X]$ decomposition for example quandles with Rig id $[8,1]$ and $[12,2]$. In this regard, we are interested in the following problem.
\begin{problem}
Classify finite connected quandles $X$ for which $(Inn(X),Inn(X)_{e})$ is Gelfand pair.
\end{problem}

\medskip \noindent \textbf{Acknowledgement.}
We thank Mandira Mondal and David Stanovský for the helpful discussions.

\bigskip

%\newpage
\bibliographystyle{amsalpha}
\bibliography{QuandleRing}

\end{document}